\documentclass{amsart}
\usepackage{geometry}
\usepackage{graphicx}
\usepackage{amssymb}
\usepackage[hyphens]{url}
\usepackage{amsmath}
\usepackage{amsthm}
\usepackage{cite}
\usepackage{mathrsfs}
\usepackage[title]{appendix}
\usepackage[utf8]{inputenc}
\usepackage[english]{babel}
\usepackage{hyperref}
\usepackage{cleveref}
\RequirePackage{textgreek}
\usepackage{enumitem}
\usepackage{bm}
\usepackage{nicefrac}
\usepackage{mathabx}
\usepackage[all,cmtip]{xy}
\usepackage{caption}
\usepackage{float}
\usepackage{tikz}
\hypersetup{linktoc=page}
\newtheorem{thm}{Theorem}[section]

\newtheorem{lemm}[thm]{Lemma}
\newtheorem{prop}[thm]{Proposition}

\newtheorem{defi}[thm]{Definition}

\newtheorem{rema}[thm]{Remark}

\newtheorem*{prop*}{Proposition}

\theoremstyle{remark}

\DeclareMathOperator{\tr}{tr}
\DeclareMathOperator{\uhp}{\mathbb{H}}

\DeclareMathOperator{\R}{\mathbb{R}}
\DeclareMathOperator{\Q}{\mathbb{Q}}
\DeclareMathOperator{\N}{\mathbb{N}}
\DeclareMathOperator{\Z}{\mathbb{Z}}

\DeclareMathOperator{\PSL}{PSL}
\DeclareMathOperator{\SO}{SO}

\title{Irrationality of the Length Spectrum}
\author{George Peterzil}
\address{Einstein Institute of Mathematics \\
The Hebrew University of Jerusalem \\ 91904 Jerusalem \\ Israel \vspace{7pt}}
\email{george.peterzil@mail.huji.ac.il}
\author{Guy Sapire}
\email{guy.sapire@mail.huji.ac.il}

\thanks{The second author was supported by the Israel Science Foundation (grant No. 2301/20).}

\begin{document}

\begin{abstract}
    It is a classical result of Dal'Bo that the length spectrum of a non-elementary Fuchsian group is non-arithmetic, namely, it generates a dense additive subgroup of $\mathbb{R}$. In this note we provide an elementary proof of an extension of this theorem: a non-elementary Fuchsian group contains two elements whose lengths are linearly independent over $\Q$, reproving a result of Prasad and Rapinchuk \cite{PrasRap}.
\end{abstract}

\maketitle

\section{Introduction}

Let $\Gamma$ be a non-elementary Fuchsian group, that is, a discrete subgroup of $\PSL_2(\R)$ which does not contain a cyclic subgroup of finite index.

\begin{defi}
    We say that $\gamma\in \Gamma\setminus\{e\}$ is \emph{hyperbolic} whenever it is diagonalizable over $\R$. Given a hyperbolic $\gamma\in \PSL_2(\R)$, we define the \emph{translation length of $\gamma$}, denoted $\ell(\gamma)$, to be $2\log(|\lambda|)$, where $\lambda$ is the eigenvalue of $\gamma$ with $|\lambda|\geq 1$. We define $L(\Gamma)$ to be the set of translation lengths of hyperbolic elements in $\Gamma$.
\end{defi}

The group $\PSL_2(\R)$ is the group of orientation-preserving isometries of the hyperbolic plane.
Under the classical correspondence between (conjugacy classes of) hyperbolic $\gamma\in \Gamma$ and closed geodesics on the orbifold $\uhp/\Gamma$, the set $L(\Gamma)$ can also be characterized as the set of lengths of closed geodesics on $\uhp/\Gamma$, where the latter is equipped with the hyperbolic metric.

The following is a well-known theorem of Dal'Bo.

\begin{thm}[Non-arithmeticity of the length spectrum, \cite{Dal'Bo1}]
\label{Fact: Non-arithmeticity}
    Let $\Gamma\leq\PSL_2(\R)$ be a non-elementary Fuchsian group. Then $L(\Gamma)$ generates a dense additive subgroup of $\R$.
\end{thm}

Non-arithmeticity has been widely applied in homogeneous dynamics, for example in proving the topological mixing of the geodesic flow and the density of horocycles originating from a horocyclic limit point of $\Gamma$.

This result has been generalized to related notions of spectra in other settings as well, see e.g. \cite{GUIVARCH2007209}, \cite{Benoist2}, \cite{Kim}. We make use of some basic results on radicals in finitely-generated field extensions of $\Q$ to give an elementary proof of the following result.

\begin{thm}
\label{Theorem: Main theorem}
    Let $\Gamma$ be a non-elementary Fuchsian group. Then $L(\Gamma)\subseteq\R$ contains two elements which are linearly independent over $\Q$.
\end{thm}

This theorem is a special case of a result by Prasad and Rapinchuk (See \cite[Remark 1]{PrasRap}), which was recently generalized in \cite{bipp2023}. In fact, it is shown in \cite[Proposition 1.8]{bipp2023} that $L(\Gamma)$ contains an infinite set of lengths which is linearly independent over $\Q$. Both theorems rely on deep results in number theory. 

It is worth noting that the results in \Cref{algebrosection} can be rephrased and proved in the arithmetic language of heights, in particular using the Northcott property of Moriwaki's height (see \cite[Section 4]{moriwaki}, and in general \cite{Silverman1986}).

\section{Proof of the Theorem}

We first recall a fundamental notion in the theory of discrete groups.

\begin{defi}
    Let $\Gamma$ be a matrix group. The \emph{trace field of $\Gamma$} is defined to be
    $$k\Gamma=\Q\left(\left\{\tr(\gamma):\gamma\in\Gamma\right\}\right)$$ 
\end{defi}

We start with the following.

\begin{lemm}
\label{Lemma: finitely generated trace field}
    The trace field of a finitely-generated matrix group $\Gamma$ is finitely generated over $\Q$.
\end{lemm}

\begin{proof}
    If $\Gamma=\langle\gamma_1,...,\gamma_n\rangle$ then the trace field is contained in the extension of $\Q$ generated by the entries of $\gamma_1,...,\gamma_n$, which is finitely-generated.
\end{proof}

Our final ingredient will be the following proposition, which follows directly from \Cref{Lemma: finitely generated trace field} and \Cref{Proposition: Finitely many per degree}, which we will prove later.

\begin{prop}
\label{Proposition: Finitely many per degree prime}
    Let $\Gamma\leq \PSL_2(\R)$ be a finitely-generated Fuchsian group, $k\Gamma$ its trace field, $r\in k\Gamma$ positive with $r\neq 1$, $K/k\Gamma$ a finitely generated extension of $k\Gamma$.
    Then there are only finitely many rational $0<q<1$ such that $r^q$ is at most quadratic over $K$.
\end{prop}

\begin{proof}[Proof of \Cref{Theorem: Main theorem}]
    Since every non-elementary $\Gamma$ contains a non-elementary finitely generated subgroup, we will assume that $\Gamma$ is finitely generated. Assume towards a contradiction that $L(\Gamma)\subseteq\ell_0\Q$ for some $\ell_0\in L(\Gamma)$, and for every $\gamma \in \Gamma$, let $q_\gamma \in \mathbb Q$ be such that $\ell(\gamma) = q_\gamma \ell_0$. Let $K=k\Gamma\left(e^{\frac{\ell_0}{2}}\right)$ be the field generated by $e^{\frac{\ell_0}{2}}$ over $k\Gamma$. We claim that for any $\gamma \in \Gamma$ we have that $e^{\frac{\{q_\gamma\}\ell_0}{2}}$ is at most quadratic over $K$, where $\{q_\gamma\} \in [0,1)$ is the fractional part of $q_\gamma$.
    First recall that for any $\gamma \in \Gamma$, by \cite[Lemma 12.1.2]{Maclachlan&Reid} we have $$\lvert\tr(\gamma)\rvert=e^{\frac{\ell(\gamma)}{2}}+e^{-\frac{\ell(\gamma)}{2}}$$
    This means that $x=e^{\frac{\ell(\gamma)}{2}}$ satisfies $x+x^{-1}\in k\Gamma$, so there is some $\lambda\in k\Gamma$ such that $x$ satisfies $x^2-x\lambda+1=0$, hence $e^{\frac{\ell(\gamma)}{2}} = e^{\frac{q_\gamma \ell_0}{2}}$ is at most quadratic over $k\Gamma$ and thus also over $K$. This gives us that $e^{\frac{\{q_\gamma\} \ell_0}{2}} = e^{\frac{q_\gamma \ell_0}{2}}\left(e^{\frac{\ell_0}{2}}\right)^{-\lfloor q_\gamma \rfloor}$ is at most quadratic over $K$, where $\lfloor q_\gamma \rfloor$ denotes the integer part of $q_\gamma$. By \Cref{Proposition: Finitely many per degree prime} with $r=e^{\frac{\ell_0}{2}}$, there is a bound $n \in \mathbb N$ on the denominators of such $q_\gamma$, hence $L(\Gamma)\subseteq \frac{1}{n!}\ell_0\Z$ and in particular $\langle L(\Gamma)\rangle \subseteq \mathbb R$ is a discrete subgroup, but this is a contradiction to \Cref{Fact: Non-arithmeticity}.
\end{proof}

\begin{rema}
    The same proof works for a general non-elementary $\Gamma\leq \SO(n,1)^+$.
\end{rema}

\section{Radicals in Finitely Generated Extensions of \texorpdfstring{$\Q$}{Q}}
\label{algebrosection}

\begin{prop}
    \label{Proposition: No High Root}
    Let $K\subseteq \R$ be a finitely generated extension of $\Q$, $r\in K$ positive with $r\neq 1$. Then there are only finitely many $n\in\N$ such that $r^{\frac{1}{n}}\in K$.
\end{prop}

\begin{proof}
    Write $L=\Q(r)$ and denote by $\widetilde L$ the relative algebraic closure of $L$ in $K$.
    Note that $\widetilde L$ is a finite extension of $L$ since it is algebraic over $L$ and a subfield of $K$, which is finitely generated.
    If $r$ is algebraic, then $\widetilde L$ is a number field (i.e a finite extension of $\Q$).
    Write $\mathcal{O}_{\tilde{L}}$ for the ring of integers of $\tilde{L}$, that is, the integral closure of $\Z$ in $\tilde{L}$, and  suppose first $r\in\mathcal{O}^\times_{\widetilde L}$. Since $r$ is integral and $\mathcal{O}_{\widetilde L}$ is integrally closed in $\widetilde L$, any element of the form $r^{\frac{1}{n}}$ is in $\mathcal{O}^\times_{\widetilde L}$. By Dirichlet's unit theorem \cite[Theorem 7.4]{Neurkich} this group is free up to roots of unity. Because $r>0$ and is different from $1$ it is not a root of unity, so it can only have finitely many roots in $\mathcal{O}_{\widetilde L}$. If $r\notin\mathcal{O}^\times_{\widetilde L}$, recall that by the unique factorization of fractional ideals \cite[Corollary 3.9]{Neurkich}, there exist prime ideals $\mathfrak p_1,\ldots, \mathfrak p_t \subseteq \mathcal O_K$ and $e_1, \ldots, e_t \in \mathbb Z \setminus \{0\}$ such that
    $$
    (r) = \mathfrak p_1^{e_1} \ldots \mathfrak p_t^{e_t}
    $$
    so for every $k \in \mathbb N$, if a $k$-th root of $r$ is in $K$ we get
    $$
    (r^{1/k})^k = (r) = \mathfrak p_1^{e_1} \ldots \mathfrak p_t^{e_t}
    $$
    and so from uniqueness necessarily $k \mid e_i$ for every $i \le t$, meaning the only roots of $r$ that could be in $K$ are $d$-th roots for $d$ divisors of $\gcd(e_1,\ldots,e_t)$.
    
    Otherwise $r$ is transcendental. Because $r$ is a free generator of $L$, the degree of $r^{\frac{1}{n}}$ over $L$ is $n$ and so for every $n > [\widetilde L:L]$ necessarily $r^{\frac{1}{n}} \notin K$.
\end{proof}

The following basic lemma will be of use.

\begin{lemm}
\label{Lemma: Degree of radical}
    Let $K\subseteq \R$ be a field, $t,n\in\N$ and $\alpha\in\R^\times$ such that $\alpha^n\in K$ and $[K(\alpha):K] = t$, then $\alpha^t \in K$.
\end{lemm}

\begin{proof}
     The minimal polynomial $p(x)$ of $\alpha$ over $K$ is of degree $t$. Since $x^n-\alpha^n=\prod_{i=1}^n\left(x-\zeta^k_n \alpha\right)$ (for $\zeta_n$ a fixed primitive root of unity) has coefficients in $K$ and has a root $\alpha$, we must have $p(x)\mid x^n-\alpha^n$, hence there are $1\leq k_1,...,k_t\leq n$ such that $p(x)=\prod_{i=1}^t\left(x-\zeta^{k_i}_n \alpha\right)$, and in particular the constant coefficient of $p(x)$ is $(-1)^t\zeta^{k_1+...+k_t}_n\alpha^t \in K$. Since $K \subseteq \R$ and $\alpha^t \in \R$, $(-1)^t\zeta^{k_1+...+k_t}_n$ is a real root of unity hence equal to $\pm 1$, so the constant coefficient of $p(x)$ is $\pm \alpha^t \in K$ and so $\alpha^t \in K$.
\end{proof}

Using this, we prove the following.

\begin{prop}
\label{Proposition: Finitely many per degree}
    Let $K\subseteq \R$ be a finitely generated extension of $\Q$, $r\in K$ positive, $r \ne 1$ and $t\in\N$. Then there are only finitely many $n\in\N$ such that $\left[K\left(r^{\frac{1}{n}}\right):K\right]= t$. Moreover, there are only finitely many rational $0<q<1$ such that $r^q$ is of degree at most $t$ over $K$.
\end{prop}

\begin{proof}
    Suppose $n \in \mathbb N$ satisfies $\left[K\left(r^{\frac{1}{n}}\right):K\right]= t$. By \Cref{Lemma: Degree of radical} for $\alpha = r^{\frac{1}{n}}$ we get $r^{\frac{t}{n}} \in K$ and by \Cref{Proposition: No High Root} there are only finitely such $n$. Let $n_0$ be minimal such that for all $n \ge n_0$ we have that $r^{\frac{1}{n}}$ is of degree strictly greater than $t$ over $K$. Assume towards a contradiction that there is some rational $0<q=\frac{m}{n}<1$ such that $r^q\in K$, with $\gcd(m,n)=1$. By B\'ezout's identity, there are integers $x,y\in\mathbb{Z}$ such that $mx+ny=1$, so in particular $r^{\frac{1}{n}}=\left(r^q\right)^x r^y \in K\left(r^q\right)$, so $[K\left(r^q\right):K] \ge \left[K\left(r^{\frac{1}{n}}\right):K\right] > t$, arriving at a contradiction.
\end{proof}

\newpage

\section*{Acknowledgments}
This paper is part of the Ph.D thesis of G.P. conducted at the Hebrew University of Jerusalem, under the supervision of Or Landesberg. G.P. would like to thank Or for suggesting this question and for reviewing multiple versions of this paper.
G.S. would like to thank Borys Kadets for enriching conversations, and Ari Shnidman for reviewing the paper. The authors would like to thank Fran\c{c}oise Dal'Bo for several insightful comments on a previous version of the manuscript, and to Rafael Potrie for pointing out the references \cite{PrasRap} and \cite{bipp2023}.

\bibliographystyle{plainurl}
\bibliography{Bibliography}

\end{document}